\documentclass[twoside]{amsart}
\usepackage[cp1251]{inputenc}
\usepackage[T2A]{fontenc}
\usepackage{array}
\usepackage{amssymb}
\usepackage{amsmath}
\usepackage{amsthm}
\usepackage{latexsym}
\usepackage{bm}
\usepackage{enumerate}
\usepackage[dvips]{graphicx}
\usepackage{epsf}
\usepackage{wrapfig}
\usepackage{euscript}
\usepackage[english,russian]{babel}
\usepackage{textcomp}
\usepackage{setspace}
\newtheorem{theorem}{Theorem}
\newtheorem{lemma}{Lemma}
\newtheorem{definition}{Definition}
\newtheorem{proposition}{Proposition}

\newtheorem{corollary}{Corollary}
\begin{document}
{\selectlanguage{english}
\binoppenalty = 10000 %
\relpenalty   = 10000 %

\pagestyle{headings} \makeatletter

\renewcommand{\@oddhead}{\raisebox{0pt}[\headheight][0pt]{\vbox{\hbox to\textwidth{{Lindstr\"{o}m theorem for intuitionistic propositional logic}\hfill \strut\thepage}\hrule}}}
\makeatother

\title{A Lindstr\"{o}m theorem for intuitionistic propositional logic}

\author{Guillermo Badia}

\address{ Department of Knowledge-Based Mathematical Systems\\
Johannes Kepler University\\
Linz, Austria 
}

\address{
School of Historical and Philosophical Inquiry\\
University of Queensland\\
Brisbane, Australia
}
\email{guillebadia89@gmail.com}

\author{Grigory Olkhovikov}

\address{ Department of Philosophy \\
Ruhr University Bochum\\
Bochum, Germany
}
\email{grigory.olkhovikov@rub.de}

\date{}
\maketitle

\begin{abstract}
It is shown that propositional intuitionistic  logic is the maximal (with respect to expressive power) abstract logic satisfying a certain topological property reminiscent of compactness, the Tarski union property and preservation under asimulations. 

\bigskip


\bigskip

\emph{Math subject classification:} 	03C95,  		03B55.
\end{abstract}
\section{Introduction}

The well-known characterization of classical first-order logic by
Per Lindstr\"{o}m (published as \cite{lin}) has had a great deal
of repercussion in contemporary logic. The key development
inspired by that paper was the introduction of a notion of
``extended first-order logic" that encompassed a great number of
expressive extensions of first order logic\footnote{This was soon
replaced by the now more common term ``abstract logic''. In fact,
probably the term ``model-theoretic language" (\cite{fe}) is more
accurate, depending on one's views of what a ``logic" is.} which let
Lindstr\"om establish, roughly, that there were no extensions of
classical first order logic that would also satisfy the
compactness and L\"owenheim-Skolem theorems. A nice accessible
exposition can be found in \cite{flum}.

 Lindstr\"om's theorem single-handedly started a new area of research known as \emph{abstract} or \emph{soft} model theory (cf. \cite{bar1, barfer}). These two adjectives are used since, when working in this field,  one finds themself using ``only very general properties of the logic, properties that carry over to a large number of other logics" (\cite{bar1}, p. 225). Some common examples of such properties are compactness, the Craig interpolation theorem or the Beth definability theorem.

Naturally, there is no reason why one could not do non-classical
abstract model theory as well, this time focusing on some
non-classical logic and its structures. An interesting survey of
this much less developed field can be found in \cite{mm}. The most
prominent work has been done in the case of modal logic. A
succession of increasingly improved results were obtained by a
number of authors such as de Rijke (\cite{rijke}), van Benthem
(\cite{vanB}), Otto and Piro (\cite{otto}) and, more recently,
Enqvist (\cite{enqvist}). Indeed,  Enqvist's result is a bit more
general than the others and it serves as inspiration for this
paper. It roughly states that the basic modal language is the most
expressive logic (over any class of relational models
axiomatizable in first order logic by a collection of strict Horn
formulas) satisfying the so called Tarski union property,
invariance under bisimilarity and compactness.

The goal of this paper is to obtain a theorem in the vein of
Enqvist but for intuitionistic models, an intuitionistic language,
replacing invariance under bisimulation for preservation under
\emph{asimulations} and compactness by a topological property more suited to a
boolean negationless context.  The notion of an asimulation was
introduced  in \cite{o}\footnote{It was an
independent rediscovery in somewhat different clothing of the
notion of an intuitionistic directed bisimulations from
\cite{kurto}.} and used to characterize the expressive power of
intuitionistic languages as fragments of first order logic over
any class of structures axiomatizable by a first order theory.
Later the notion of asimulation was modified to capture also
expressive powers of intuitionistic first-order logic \cite{o1},
various systems of basic intuitionistic modal logic \cite{o2}, and
also some systems which are not connected with intuitionistic
logic at all \cite{o3}.

Our work will be of interest not only because it deals with
probably the second most famous non-classical logic but because it
will be a contribution to abstract model theory without Boolean
negation, a subfield where not many results are known.

The layout of the remaining part of the present paper is then as
follows. In Section \ref{S:Prel} we establish the main bits of our
notation and define some basic concepts. Section \ref{S:Unravel}
is then devoted to proving some less immediate lemmas on
intuitionistic unravellings and intuitionistically saturated
models. In Section \ref{S:Abstract} we define the notion of an
abstract intuitionistic logic and formulate our main result. We prove such result in the next section. In Section \ref{S:Top} we provide 
topological interpretation of the main theorem. We end the paper  by
 suggesting some lines of further inquiry.

\section{Preliminaries}\label{S:Prel}

In this paper, we consider the language of intuitionistic
propositional logic, which we identify with its set of formulas.
This language is generated from some set of propositional letters
by a finite number of applications of connectives from the set $\{
\bot, \top, \wedge, \vee, \to \}$; the connectives are assumed to
have they usual arities. The set of propositional letters can be
in general arbitrary large, but we assume that it is disjoint from
the above set of connectives (and from the set of logical symbols
of every logic which we are going to consider below). Any set with
this property we will call \emph{vocabulary}. Intuitionistic
propositional formulas will be denoted with Greek letters like
$\varphi$, $\psi$ and $\theta$,\footnote{We will be adjoining to
them subscripts and superscripts when needed, and the same is
assumed for any other notations introduced in this paper.} whereas
the elements of vocabularies will be denoted by letters $p$, $q$,
and $r$. If $\Theta$ is a vocabulary, then $IL(\Theta)$ denotes
the set of those intuitionistic propositional formulas which only
contain propositional letters from $\Theta$.

For this language, we assume the standard Kripke semantics.
The typical notations for intuitionistic Kripke models that we are
going to use below are as follows:
$$
\mathcal{M} = \langle W, R, V\rangle, \mathcal{M}' = \langle W',
R', V'\rangle, \mathcal{M}_n = \langle W_n, R_n, V_n\rangle,
$$
$$
\mathcal{N} = \langle U, S, Y\rangle, \mathcal{N}' = \langle U',
S', Y'\rangle,\mathcal{N}_n = \langle U_n, S_n, Y_n\rangle,
$$
where $n \in \omega$. As we proceed, we will also define some
operations on intuitionistic Kripke models which are going to
affect the notation.

If $\Theta$ is a vocabulary and $\mathcal{M}$ is an intuitionistic
Kripke $\Theta$-model, then $W$ is a non-empty set of
\emph{worlds}, \emph{states}, or \emph{nodes}, $R$ is a pre-order
on $W$ called $\mathcal{M}$'s \emph{accessibility relation}, and
$V$ is the evaluation function for $\Theta$ in $\mathcal{M}$, that
is to say, a function $V:\Theta \to 2^W$ such that for
every $p \in \Theta$ and arbitrary $s, t \in W$, it is true that:
$$
w\mathrel{R}v \Rightarrow (w \in V(p) \Rightarrow v \in V(p)).
$$

If $\mathcal{M}$ and $\mathcal{N}$ are two intuitionistic Kripke
$\Theta$-models then we say that $\mathcal{M}$ is a
\emph{submodel} of $\mathcal{N}$ and write $\mathcal{M} \subseteq
\mathcal{N}$ iff $W \subseteq U$, $R = S\upharpoonright(W\times
W)$ and, for every $p \in \Theta$, $V(p) = Y(p) \cap W$. In
general, for every $W \subseteq U$ there exists a corresponding
submodel  $\mathcal{M}$ of $\mathcal{N}$ with $W$ as its universe.
In such cases we may also denote $\mathcal{M}$ by
$\mathcal{N}(W)$.

If $\mathcal{M}_1 \subseteq,\ldots, \subseteq \mathcal{M}_n
\subseteq,\ldots$ is a countable chain of intuitionistic Kripke
models then the model:
$$
\bigcup_{n \in \omega}\mathcal{M}_n = (\bigcup_{n \in \omega}W_n,
\bigcup_{n \in \omega}R_n, \bigcup_{n \in \omega}V_n)
$$
is again an intuitionistic Kripke model.

A \emph{pointed} intuitionistic Kripke $\Theta$-model is a pair of
the form $(\mathcal{M}, w)$ such that $w \in W$. In this paper, we
are not going to consider any non-intuitionistic Kripke models.
Therefore, we will omit the qualification `intuitionistic Kripke'
in what follows and will simply speak about (pointed)
$\Theta$-models. We assume the standard satisfaction relation for
$IL(\Theta)$:
\begin{align*}
&\mathcal{M}, w \models_{IL} p_n \Leftrightarrow w \in V(p_n),
&&\textup{for }n \in \mathbb{N};\\
&\mathcal{M}, w \models_{IL} \varphi \wedge \psi \Leftrightarrow
\mathcal{M}, w \models_{IL} \varphi\textup{ and } \mathcal{M}, w
\models_{IL} \psi;\\
&\mathcal{M}, w \models_{IL} \varphi \vee \psi \Leftrightarrow
\mathcal{M}, w \models_{IL} \varphi\textup{ or } \mathcal{M}, w
\models_{IL} \psi;\\
&\mathcal{M}, w \models_{IL} \varphi \to \psi \Leftrightarrow
\forall v(w\mathrel{R}v \Rightarrow (\mathcal{M}, v
\not\models_{IL} \varphi\textup{ or } \mathcal{M}, v
\models_{IL} \psi));\\
&\mathcal{M}, w \not\models_{IL} \bot;\\
&\mathcal{M}, w \models_{IL} \top.
\end{align*}
note that the formulas of $IL(\Theta)$ get satisfied at pointed
$\Theta$-models rather than at models alone.

Since classical negation is not available in intuitionistic logic,
intuitionistic theories are often defined to include falsehood
assumptions along with truth assumptions (this is done in e.g.
\cite[p. 110]{GabbayMaksimova} w.r.t. intuitionistic first-order
logic). Thus an $IL(\Theta)$-theory becomes a pair $(\Gamma,
\Delta) \in 2^{IL(\Theta)}\times 2^{IL(\Theta)}$, where formulas
in $\Gamma$ are assumed to be true and formulas from $\Delta$ are
assumed to be false. If $(\mathcal{M}, w)$ is a pointed
$\Theta$-model, then we define $Th_{IL}(\mathcal{M}, w)$, the
\emph{$IL(\Theta)$-theory of} $(\mathcal{M}, w)$, as follows:
$$
Th_{IL}(\mathcal{M}, w) := (\{ \varphi \in IL(\Theta)\mid
\mathcal{M}, w \models_{IL} \varphi \}, \{ \varphi \in
IL(\Theta)\mid \mathcal{M}, w \not\models_{IL} \varphi \}).
$$
We also introduce a special notation for the left and right
projection of $Th_{IL}(\mathcal{M}, w)$, that is to say, for the
\emph{positive} and for the \emph{negative} part of this theory,
denoting them by $Th^+_{IL}(\mathcal{M}, w)$ and
$Th^-_{IL}(\mathcal{M}, w)$, respectively. Inclusion of
intuitionistic theories must then involve set-theoretic inclusion
of their respective projections, so that we define:
$$
(\Gamma, \Delta) \subseteq (\Gamma', \Delta') \Leftrightarrow
\Gamma \subseteq \Gamma'\textup{ and }\Delta \subseteq \Delta'.
$$
It is clear then that an $IL(\Theta)$-theory $(\Gamma, \Delta)$ is
\emph{$IL$-satisfiable} iff we have $(\Gamma, \Delta) \subseteq
Th_{IL}(\mathcal{M}, w)$ for some pointed $\Theta$-model
$(\mathcal{M}, w)$. In this case we will also write $\mathcal{M},
w \models_{IL} (\Gamma, \Delta)$. If $\mathcal{M} \subseteq
\mathcal{N}$ and for every $w \in W$ it is true that
$Th_{IL}(\mathcal{M}, w) = Th_{IL}(\mathcal{N}, w)$, then we say
that $\mathcal{M}$ is an $IL$\emph{-elementary submodel} of
$\mathcal{N}$ and write $\mathcal{M} \preccurlyeq_{IL}
\mathcal{N}$.

We end this section with some definitions and brief discussion of
asimulations and some other related notions relevant to the
subject of this paper.
\begin{definition}\label{D:asimulation}
{\em Let $(\mathcal{M}_1, w_1)$, $(\mathcal{M}_2, w_2)$ be pointed
$\Theta$-models. A binary relation $A$ is called an
\emph{asimulation from $(\mathcal{M}_1,w_1)$ to
$(\mathcal{M}_2,w_2)$} iff for any $i,j$ such that $\{ i,j \} = \{
1, 2 \}$, any $v \in W_i$, $s,t \in W_j$, any propositional letter
$p \in \Theta$ the following conditions hold:
\begin{align}
&A \subseteq (W_1 \times W_2) \cup (W_2\times
W_1)\label{E:c22}\tag{\text{s-type}}\\
&w_1\mathrel{A}w_2\label{E:c11}\tag{\text{elem}}\\
&(v\mathrel{A}s \wedge  v \in V_i(p)) \Rightarrow s \in V_j(p)))\label{E:c33}\tag{\text{s-atom}}\\
&(v\mathrel{A}s \wedge s\mathrel{R_j}t) \Rightarrow \exists u \in
W_i(v\mathrel{R_i}u \wedge t\mathrel{A}u \wedge
u\mathrel{A}t)\label{E:c44}\tag{\text{s-back}}
\end{align}
}
\end{definition}

Intuitionistic propositional formulas are known to be preserved
under asimulations. More precisely, if $(\mathcal{M}_1, w_1)$, and
$(\mathcal{M}_2, w_2)$ are pointed $\Theta$-models and $A$ is an
asimulation from $(\mathcal{M}_1,w_1)$ to $(\mathcal{M}_2,w_2)$,
then $Th^+_{IL}(\mathcal{M}_1, w_1) \subseteq
Th^+_{IL}(\mathcal{M}_2, w_2)$. Moreover, preservation under
asimulations is known to semantically characterize $IL$ as a
fragment of classical first-order logic, see \cite{o} for the
proof.

Asimulations are defined as an intuitionistic version of
bisimulations, a well-known  concept from classical modal logic.
In fact, bisimulations can be even defined as symmetric
asimulations. It follows then that if $B$ is a bisimulation
between $(\mathcal{M}_1,w_1)$ and $(\mathcal{M}_2,w_2)$ then $B$
is an asimulation both from $(\mathcal{M}_1,w_1)$ to
$(\mathcal{M}_2,w_2)$ and from $(\mathcal{M}_2,w_2)$ to
$(\mathcal{M}_1,w_1)$. Further, given a pair of pointed
$\Theta$-models $(\mathcal{M}_1, w_1)$ and $(\mathcal{M}_2, w_2)$
and a bisimulation $B$ between them, we get that
$Th_{IL}(\mathcal{M}_1, w_1) = Th_{IL}(\mathcal{M}_2, w_2)$.

Asimulation is not the only concept from which bisimulation arises
as its symmetric version; another (and more traditional) notion of
this kind would be simulation which is defined as follows:
\begin{definition}\label{D:simulation}
{\em Let $(\mathcal{M}_1, w_1)$, $(\mathcal{M}_2, w_2)$ be pointed
$\Theta$-models. A binary relation $Z$ is called an
\emph{simulation from $(\mathcal{M}_1,w_1)$ to
$(\mathcal{M}_2,w_2)$} iff for any $v, u \in W_1$, $s \in W_2$,
any propositional letter $p \in \Theta$ the condition
\eqref{E:c11} holds together with the following conditions:
\begin{align}
&Z \subseteq (W_1 \times W_2)\label{E:c2}\tag{\text{type}}\\
&(v\mathrel{Z}s \wedge  v \in V_1(p)) \Rightarrow s \in V_2(p))\label{E:c3}\tag{\text{atom}}\\
&(v\mathrel{Z}s \wedge v\mathrel{R_1}u) \Rightarrow \exists t \in
W_2(s\mathrel{R_2}t \wedge u\mathrel{Z}t
)\label{E:c4}\tag{\text{forth}}
\end{align}
}
\end{definition}
Note that in the definition of asimulation one finds
\emph{stronger} versions of clauses \eqref{E:c2} and \eqref{E:c3},
hence the labels \eqref{E:c22} and \eqref{E:c33}. The label
\eqref{E:c44} also indicates a strengthening of a more traditional
back-clause (given that the symmetry of $A$ is not guaranteed).

Simulations are clearly related to homomorphisms, which we define,
a la \cite{BdRV}, as follows:
\begin{definition}\label{D:homomorphism}
{\em Let $\mathcal{M}_1$, $\mathcal{M}_2$ be $\Theta$-models. A
function $h:W_1 \to W_2$ is called a \emph{homomorphism from
$\mathcal{M}_1$ to $\mathcal{M}_2$} iff for any $v, u \in W_1$,
any propositional letter $p \in \Theta$ the following conditions
hold:
\begin{align}
&v \in V_1(p) \Rightarrow h(v) \in V_2(p)\label{E:c3h}\tag{\text{atom}}\\
&v\mathrel{R_1}u \Rightarrow h(v)\mathrel{R_2}h(u)
\label{E:c4h}\tag{\text{hom}}
\end{align}
}
\end{definition}
It is clear that the main difference between simulations and
homomorphisms is that the latter are functions while the former
are relations. We will see in the next section that a sort of
connection between the two can be established and this fact is
important for our result.

A stronger version of asimulation in the context of intuitionistic
logic would be an $IL$-embedding defined as follows:
\begin{definition}\label{D:embedding}
{\em Let $\mathcal{M}_1$, $\mathcal{M}_2$ be $\Theta$-models. A
function $h:W_1 \to W_2$ is called an \emph{ $IL$-embedding of
$\mathcal{M}_1$ into $\mathcal{M}_2$} iff $h$ is an injective
homomorphism from $\mathcal{M}_1$ to $\mathcal{M}_2$, and for all
$v,u \in W_1$ it is true that
\begin{align}
&Th_{IL}(\mathcal{M}_1, v) = Th_{IL}(\mathcal{M}_2, h(v))\label{E:c3e}\tag{\text{theories}}\\
&v\mathrel{R_1}u \Leftrightarrow h(v)\mathrel{R_2}h(u)
\label{E:c4e}\tag{\text{s-hom}}
\end{align}
 }
\end{definition}
It is easy to notice that whenever $f$ is an $IL$-embedding of
$\mathcal{M}_1$ into $\mathcal{M}_2$, we have that
$\mathcal{M}_2(f(W)) \preccurlyeq_{IL} \mathcal{M}_2$. Also
$\mathcal{M}_1$ is obviously isomorphic to $\mathcal{M}_2(f(W))$,
so that replacing $\mathcal{M}_2(f(W))$ with a copy of
$\mathcal{M}_1$ changes nothing. Therefore, whenever there exists
is an $IL$-embedding of $\mathcal{M}_1$ into $\mathcal{M}_2$ we
may assume that $\mathcal{M}_1 \preccurlyeq_{IL} \mathcal{M}_2$

\section{Intuitionistic unravellings and saturated models}\label{S:Unravel}

In this section, we treat some less immediate properties of
intuitionistic logic. We start by introducing a further piece of
notation. If $(o_1,\ldots, o_n)$ is an $n$-tuple of objects of any
nature, then we will denote it with $\bar{o}_n$.

Unravelling Kripke models is another well-known item from
classical modal logic. Here we adapt it to the intuitionistic
setting. For a given intuitionistic pointed Kripke $\Theta$-model
$(\mathcal{M}, w)$, the model $\mathcal{M}^{un}_w = \langle
W^{un}_w, R^{un}_w, V^{un}_w\rangle$, called the intuitionistic
unravelling of $\mathcal{M}$ around $w$ is defined as follows.
\begin{itemize}
\item $W^{un}_w = \{ \bar{u}_n \in W^n\mid u_1 = w,\,(\forall i <
n)(w_i\mathrel{R}w_{i + 1}) \}$;

\item $R^{un}_w$ is the reflexive and transitive closure of the
following relation: $\{ (s,t) \in W^2 \mid (\exists u \in W)(t =
(s, u)) \}$;

\item For arbitrary $p \in \Theta$, we have $V^{un}_w(p) = \{
\bar{u}_n \in W^{un}_w \mid u_n \in V(p) \}$.
\end{itemize}

The following lemma sums up the basic facts about intuitionistic
unravellings:
\begin{lemma}\label{unravellinglemma}
Let $(\mathcal{M}, w)$ be a pointed $\Theta$-model. Then:
\begin{enumerate}
\item $\mathcal{M}^{un}_w$ is a $\Theta$-model;

\item $R^{un}_w$ is antisymmetric;

\item $(\mathcal{M}^{un}_w, w)$ is bisimilar to $(\mathcal{M},
w)$;

\item $Th_{IL}(\mathcal{M}, v_k) = Th_{IL}(\mathcal{M},
\bar{v}_k)$ for any $\bar{v}_k \in W^{un}_w$;

\item $Th_{IL}(\mathcal{M}, \bar{u}_n) = Th_{IL}(\mathcal{M},
\bar{v}_k)$ for any $\bar{u}_n, \bar{v}_k \in W^{un}_w$ whenever
$u_n = v_k$.
\end{enumerate}
\end{lemma}
\begin{proof} By definition, $R^{un}_w$ is reflexive and transitive. Moreover,
we note that for arbitrary $\bar{w}_k, \bar{v}_n \in W^{un}_w$ we
have
$$
\bar{w}_k\mathrel{R^{un}_w}\bar{v}_n \Leftrightarrow k \leq n
\wedge \bar{w}_k = \bar{v}_k.
$$
Therefore, to show antisymmetry, assume that for a given
$\bar{w}_k, \bar{v}_n \in W^{un}_w$ we have

$$
\bar{w}_k\mathrel{R^{un}_w}\bar{v}_n \wedge
\bar{v}_n\mathrel{R^{un}_w}\bar{w}_k.
$$
By the above biconditional it immediately follows that $k \leq n
\wedge n \leq k$ so that we have $k = n$ and, further, that
$\bar{w}_k = \bar{v}_k$. Therefore, we get $\bar{w}_k = \bar{v}_n$
and thus $R^{un}_w$ is shown to be antisymmetric.

To show monotonicity of $V^{un}_w$ w.r.t. $R^{un}_w$, assume that
for some $\bar{w}_k, \bar{v}_n \in W^{un}_w$ we have
$\bar{w}_k\mathrel{R^{un}_w}\bar{v}_n$. Then we have both $k \leq
n$ and $\bar{w}_k = \bar{v}_k$, so that, in particular, we get
$w_k = v_k$. Further, by definition of $W^{un}_w$ we have that
$$
v_k\mathrel{R}v_{k + 1}\mathrel{R},\ldots,\mathrel{R}v_{n -
1}\mathrel{R}v_n,
$$
whence, by monotonicity of $V$ w.r.t. $R$, it follows that for
every $p \in \Theta$, if $v_k \in V(p)$, then $v_n \in V(p)$. It
remains to notice that, by definition of $V^{un}_w$ we have both
$$
\bar{w}_k \in V^{un}_w(p) \Leftrightarrow w_k \in V(p)
\Leftrightarrow v_k \in V(p),
$$
(by $w_k = v_k$) and
$$
\bar{v}_n \in V^{un}_w(p) \Leftrightarrow v_n \in V(p),
$$
whence we obtain that $\bar{v}_n \in V^{un}_w(p)$, whenever
$\bar{w}_k \in V^{un}_w(p)$.

As for the bisimulation $B$ between $(\mathcal{M}^{un}_w, w)$ and
$(\mathcal{M}, w)$, it can be defined as follows:
$$
B := \{(w_n, \bar{w}_n), (\bar{w}_n,w_n) \mid (\exists k \geq
0)(w\mathrel{R^{un}_w} w_n)\}.
$$
Parts $4$ and $5$ then follow from this definition of bisimulation
and preservation of intuitionistic formulas under asimulations.
\end{proof}
We state here one more lemma on intuitionistic unravellings
relevant to our main result:
\begin{lemma}\label{L:homomorphism}
Let $(\mathcal{M}, w)$ and $(\mathcal{N}, u)$ be  pointed
$\Theta$-models, and let $Z$ be a simulation from
$(\mathcal{M}^{un}_w, w)$ to $(\mathcal{N}, u)$ such that the left
projection of $Z$ is $W^{un}_w$ (that is to say, $Z$ is total).
Then there exists a homomorphism $h: \mathcal{M}^{un}_w \to
\mathcal{N}$ such that $h(w) = u$ and $v\mathrel{Z}h(v)$ for every
$v \in W^{un}_w$.
\end{lemma}
\begin{proof}
Suppose that there is total simulation $Z$ from
$(\mathcal{M}^{un}_w, w)$ to $(\mathcal{N}, u)$. We then define
$h$ by induction on the length of $\bar{v}_n \in W^{un}_w$. The
base case is when $\bar{v}_n = w$ and we set $h(w) := u$. If $n =
k + 1$, then $h(\bar{v}_k)$ must be already defined by induction
hypothesis, and we must also have
$\bar{v}_k\mathrel{Z}h(\bar{v}_k)$. Also, we have
$\bar{v}_k\mathrel{R^{un}_w}\bar{v}_{k + 1}$. But then by
condition \eqref{E:c4}, there must be some $s \in U$ such that
$h(\bar{v}_k)\mathrel{Z}s$ and $\bar{v}_{k + 1}\mathrel{Z}s$.
Choose one such $s$ and set $h(\bar{v}_{k + 1}) := s$. Since
$\bar{v}_k$, the immediate predecessor of $\bar{v}_{k + 1}$ different from $\bar{v}_{k + 1}$  in
$\mathcal{M}^{un}_w$ is unique (naturally, there might be other predecessors of $\bar{v}_{k + 1}$ but they have to lay below $\bar{v}_{k }$), $h$ is a function with the
required properties.
\end{proof}

Let $(\Gamma, \Delta)$ be an $IL(\Theta)$-theory and
$\mathcal{M}$ be  a $\Theta$-model. If $v \in W$, we will say that
$(\Gamma, \Delta)$ \emph{is finitely $IL$-satisfiable in
$\mathcal{M}$ by successors of $v$} iff for all finite $\Gamma'
\subseteq \Gamma$ and $\Delta' \subseteq \Delta$, there is a $u
\in W$ such that $v\mathrel{R}u$ and $\mathcal{M},u
\models_{IL}(\Gamma', \Delta')$. A $\Theta$-model $\mathcal{M}$ is
called \emph{$IL$-saturated} iff for every $v \in W$ and for every
$IL(\Theta)$-theory $(\Gamma, \Delta)$, whenever $(\Gamma,
\Delta)$ is finitely $IL$-satisfiable in $\mathcal{M}$ by
successors of $v$, then for some $u \in W$ it is true that
$v\mathrel{R}u$ and $\mathcal{M},u \models_{IL}(\Gamma, \Delta)$.
The importance of intuitionistically saturated models is that
among them asimulations can be defined in the following easy and
natural way:

\begin{lemma}\label{L:asimulations}
Let $(\mathcal{M}_1, w_1)$, $(\mathcal{M}_2, w_2)$ be pointed
$\Theta$-models. If $Th^+_{IL}(\mathcal{M}_1, w_1) \subseteq
Th^+_{IL}(\mathcal{M}_2, w_2)$ and both $\mathcal{M}_1$ and
$\mathcal{M}_2$ are intuitionistically saturated, then the
relation $A$ such that for all $u \in W_i$, $s \in W_j$ if $\{ i,j
\} = \{ 1,2 \}$, then
$$
u\mathrel{A}s \Leftrightarrow (Th^+_{IL}(\mathcal{M}_i, u)
\subseteq Th^+_{IL}(\mathcal{M}_j, s))
$$
is an asimulation from $(\mathcal{M}_1, w_1)$ to $(\mathcal{M}_2,
w_2)$.
\end{lemma}
\begin{proof}
The relation $A$, as defined in the lemma, obviously satisfies
conditions \eqref{E:c22}, \eqref{E:c11}, and \eqref{E:c33} given
in Definition \ref{D:asimulation}. We check the remaining
condition, \eqref{E:c44}.

Assume that $u\mathrel{A}s$, so that $Th^+_{IL}(\mathcal{M}_i, u)
\subseteq Th^+_{IL}(\mathcal{M}_j, s)$, and let for some $t \in
W_j$ we have $s\mathrel{R_j}t$. Let $(\Gamma', \Delta') \subseteq
Th_{IL}(\mathcal{M}_j, t)$ be finite. Then, of course, we have:
$$
\mathcal{M}_j, s \not\models_{IL} \bigwedge\Gamma' \to
\bigvee\Delta',
$$
and, by $u\mathrel{A}s$:
$$
\mathcal{M}_i, u \not\models_{IL} \bigwedge\Gamma' \to
\bigvee\Delta'.
$$
The latter means that $(\Gamma', \Delta') $ must be $IL$-satisfied
by some $R_i$-successor of $u$. Therefore, by intuitionistic
saturation of both $\mathcal{M}_1$ and $\mathcal{M}_2$, we get
that for some $v \in W_i$ such that $u\mathrel{R_i}w$ the theory
$Th_{IL}(\mathcal{M}_j, t)$ is $IL$-satisfied at $(\mathcal{M}_i,
u)$. It follows immediately that
$$
Th_{IL}(\mathcal{M}_i, w) = Th_{IL}(\mathcal{M}_j, t),
$$
whence
$$
Th^+_{IL}(\mathcal{M}_i, w) = Th^+_{IL}(\mathcal{M}_j, t),
$$
which, in turn, means that both $w\mathrel{A}t$ and
$t\mathrel{A}w$.
\end{proof}

\section{Abstract intuitionistic logics}\label{S:Abstract}
An abstract intuitionistic logic $\mathcal{L}$ is a pair $(L,
\models_\mathcal{L})$, where $L$ maps every vocabulary $\Theta$ to
the set $L(\Theta)$ of $\Theta$-formulas of $\mathcal{L}$ and
$\models_\mathcal{L}$ is a binary relation between pointed models
and elements of $L(\Theta)$ for some vocabulary $\Theta$ such that
the following conditions are satisfied:
\begin{itemize}
\item $\Theta \subseteq \Theta' \Rightarrow L(\Theta) \subseteq
L(\Theta')$.

\item If $\mathcal{M}$ is a $\Theta$-model and $\mathcal{M},w
\models_\mathcal{L} \phi$, then $\phi \in L(\Theta)$.

\item If $\mathcal{M}$ and $\mathcal{N}$ are $\Theta$-models,
$\phi \in L(\Theta)$ and $f$ is an isomorphism between
$\mathcal{M}$ and $\mathcal{N}$, then for every $w \in W$ it is
true that:
$$
\mathcal{M}, w \models_\mathcal{L} \phi \Leftrightarrow
\mathcal{N}, f(w) \models_\mathcal{L} \phi.
$$

\item (Expansion). If $\Theta$ is a vocabulary, $\phi \in
L(\Theta)$, $\Theta \subseteq \Theta'$, $\mathcal{M}$ is a
$\Theta'$-model, and $\mathcal{M}\upharpoonright\Theta$ is the
reduct of $\mathcal{M}$ to $\Theta$, then:
$$
\mathcal{M}, w \models_\mathcal{L} \phi \Leftrightarrow
\mathcal{M}\upharpoonright\Theta, w \models_\mathcal{L} \phi.
$$

\item (Occurrence). If $\phi \in L(\Theta)$ for some vocabulary ,
then there is a finite $\Theta_\phi \subseteq \Theta$ such that
for every $\Theta'$-model $\mathcal{M}$, the relation
$\mathcal{M}\models_\mathcal{L}\phi$ is defined iff $\Theta_\phi
\subseteq \Theta'$.

\item (Closure). For every vocabulary $\Theta$ and all $\phi, \psi
\in L(\Theta)$, we have $\phi \to \psi, \phi \wedge \psi, \phi
\vee \psi \in L(\Theta)$, that is to say, $\mathcal{L}$ is closed
under intuitionistic implication, conjunction and disjunction.
\end{itemize}

We further define that given a pair of abstract intuitionistic
logics $\mathcal{L}$ and $\mathcal{L}'$, we say that
$\mathcal{L}'$ extends $\mathcal{L}$ and write $\mathcal{L}
\trianglelefteq \mathcal{L}'$ when for all vocabularies $\Theta$
and $\phi\in L(\Theta)$ there exists a $\psi\in L'(\Theta)$ such
that for arbitrary pointed $\Theta$-model $(\mathcal{M}, w)$ it is
true that:
$$
\mathcal{M}, w\models_\mathcal{L}\phi \Leftrightarrow \mathcal{M},
w\models_{\mathcal{L}'}\psi.
$$
If both $\mathcal{L} \trianglelefteq \mathcal{L}'$ and
$\mathcal{L}' \trianglelefteq \mathcal{L}$ holds, then we say that
the logics $\mathcal{L}$ and $\mathcal{L}'$ are \emph{expressively
equivalent} and write $\mathcal{L} \equiv \mathcal{L}'$.

It is easy to see that intuitionistic propositional logic itself
turns out to be an abstract intuitionistic logic $\mathsf{IL} =
(IL, \models_{IL})$ under this definition. It is also obvious that
the above definitions and conventions about intuitionistic
theories can be carried over to an arbitrary abstract
intuitionistic logic $\mathcal{L}$ replacing everywhere $IL$ with
$\mathcal{L}$, including such notions as elementary submodel,
embedding, saturation of a model, etc. In particular, since the
relation $\models_\mathcal{L}$ never distinguishes between
isomorphic models, the remark after Definition \ref{D:embedding}
holds for arbitrary intuitionistic logics $\mathcal{L}$.

In this paper, our specific interest is in the extensions of
$\mathsf{IL}$. Since every abstract intuitionistic logic
$\mathcal{L}$ extending $\mathsf{IL}$ must have an equivalent for
every intuitionistic propositional formula, we will just assume
that for every vocabulary $\Theta$ we have $IL(\Theta) \subseteq
L(\Theta)$ and that for every $\varphi \in IL(\Theta)$ and every
pointed $\Theta$-model $(\mathcal{M}, w)$ we have that:
$$
\mathcal{M}, w\models_\mathcal{L}\varphi \Leftrightarrow
\mathcal{M}, w\models_{\mathsf{IL}}\varphi,
$$
so that all the intuitionistic
propositional formulas are present in $\mathcal{L}$ in their usual
form and with their usual meaning, and whatever other formulas
that $\mathcal{L}$ may contain are distinct from the elements of
$IL(\Theta)$.

We can immediately state the following corollary to Lemma
\ref{L:asimulations} for arbitrary extensions of $\mathsf{IL}$:

\begin{corollary}\label{L:asimulationscorollary}
Let $\mathsf{IL} \trianglelefteq \mathcal{L}$, and let
$(\mathcal{M}_1, w_1)$, $(\mathcal{M}_2, w_2)$ be two pointed
intuitionistic Kripke $\Theta$-models. If
$Th^+_{IL}(\mathcal{M}_1, w_1) \subseteq Th^+_{IL}(\mathcal{M}_2,
w_2)$ and both $\mathcal{M}_1$ and $\mathcal{M}_2$ are
$\mathcal{L}$-saturated, then the relation $A$ such that for all
$u \in W_i$, $s \in W_j$ if $\{ i,j \} = \{ 1,2 \}$, then
$$
u\mathrel{A}s \Leftrightarrow (Th^+_{IL}(\mathcal{M}_i, u)
\subseteq Th^+_{IL}(\mathcal{M}_j, s))
$$
is an asimulation from $(\mathcal{M}_1, w_1)$ to $(\mathcal{M}_2,
w_2)$.
\end{corollary}
To prove this, we just repeat the proof of Lemma
\ref{L:asimulations} using the fact that every
$\mathcal{L}$-saturated model is of course $IL$-saturated.

Some of the extensions of $\mathsf{IL}$ turn out to be better than
others in that they have useful model-theoretic properties. We
define some of the relevant properties below.
\begin{definition}\label{D:properties}
Let $\mathcal{L} = (L, \models_\mathcal{L})$ be an abstract
intuitionistic logic. Then:
\begin{itemize}
\item $\mathcal{L}$ is \textbf{invariant under asimulations}, iff
for all vocabularies $\Theta$ and arbitrary pointed
$\Theta$-models $(\mathcal{M}_1, w_1)$ and $(\mathcal{M}_2, w_2)$,
whenever $A$ is an asimulation from $(\mathcal{M}_1, w_1)$ to
$(\mathcal{M}_2, w_2)$, then the inclusion
$Th^+_\mathcal{L}(\mathcal{M}_1, w_1) \subseteq
Th^+_\mathcal{L}(\mathcal{M}_2, w_2)$ holds.

\item $\mathcal{L}$ is \textbf{intuitionistically compact}, iff an arbitrary
$L(\Theta)$-theory  $(\Gamma, \Delta)$ is
$\mathcal{L}$-satisfiable, whenever for every finite $\Gamma'
\subseteq \Gamma$ and $\Delta' \subseteq \Delta$, the theory
$(\Gamma', \Delta')$ is $\mathcal{L}$-satisfiable.

\item $\mathcal{L}$ has \textbf{Tarski Union Property (TUP)} iff
for every $\mathcal{L}$-elementary chain

$$
\mathcal{M}_0 \preccurlyeq_\mathcal{L},\ldots,
\preccurlyeq_\mathcal{L} \mathcal{M}_n
\preccurlyeq_\mathcal{L},\ldots
$$
it is true that:
$$
\mathcal{M}_n \preccurlyeq_\mathcal{L} \bigcup_{n \in
\omega}\mathcal{M}_n
$$

for all $n \in \omega$.
\end{itemize}
\end{definition}
We have mentioned above that in the case of $\mathsf{IL}$
invariance under asimulations implies invariance under
bisimulations. The same argument holds in the case of arbitrary
abstract intuitionistic logic $\mathcal{L}$. In other words, the
following lemma holds:
\begin{lemma}\label{L:bisimulationinvariance}
If $\mathcal{L}$ is an abstract intuitionistic logic that is
invariant under asimulations, then $\mathcal{L}$ is invariant
under bisimulations. In other words, for arbitrary vocabulary
$\Theta$ and arbitrary pointed $\Theta$-models $(\mathcal{M}_1,
w_1)$ and $(\mathcal{M}_2, w_2)$, if $B$ is bisimulation between
$(\mathcal{M}_1, w_1)$ and $(\mathcal{M}_2, w_2)$, then:
$$
Th_\mathcal{L}(\mathcal{M}_1, w_1) = Th_\mathcal{L}(\mathcal{M}_2,
w_2).
$$
\end{lemma}

Moreover if an abstract intuitionistic logic is invariant under
asimulations, then every formula of this logic is monotonic w.r.t.
accessibility relation. More precisely, the following lemma holds:
\begin{lemma}\label{L:monotonicity}
If $\mathcal{L}$ is an abstract intuitionistic logic that is
invariant under asimulations, then for every pointed
$\Theta$-model $(\mathcal{M}, w)$, for every $v \in W$ such that
$w\mathrel{R}v$, and for every $\phi \in L(\Theta)$ it is true
that:
$$
\mathcal{M}, w \models_\mathcal{L} \phi \Rightarrow \mathcal{M}, v
\models_\mathcal{L} \phi.
$$
\end{lemma}
\begin{proof}
We define asimulation $A$ from $(\mathcal{M}, w)$ to
$(\mathcal{M}, v)$ setting:
$$
A:= \{ (w, v) \} \cup \{ (u,u) \mid u \in W, v\mathrel{R}u \}.
$$
The lemma then follows by asimulation invariance of $\mathcal{L}$.
\end{proof}
Note that for Lemma \ref{L:monotonicity} one does not even need to
assume that $\mathsf{IL} \trianglelefteq \mathcal{L}$.
The next lemma sums up some well-known facts about
$\mathsf{IL}$:
\begin{lemma}\label{L:properties}
$\mathsf{IL}$ is invariant under asimulations, intuitionistically compact and
has TUP.
\end{lemma}
\begin{proof}
Invariance under asimulations follows from the main result of
\cite{o}. The intuitionistic compactness of $\mathsf{IL}$ is a well-known
fact and is established in many places, see e.g. \cite[Theorem
2.46]{chagrov}. The proof of TUP runs along the lines of standard
proof of TUP for modal formalisms (cf \cite[Observation 5]{otto});
we briefly sketch it here.

We need to show that if $\varphi \in IL(\Theta)$, $n \in \omega$
and $w \in W_n$, then
$$
\mathcal{M}_n, w \models_{IL} \varphi \Leftrightarrow \bigcup_{n
\in \omega}\mathcal{M}_n , w \models_{IL} \varphi;
$$
this is done by induction on $\varphi$ and the only non-trivial
case is when $\varphi = \psi \to \chi$. If $\mathcal{M}_n, w
\not\models_{IL} \psi \to \chi$, then there is a $v \in W_n$ such
that $w\mathrel{R_n}v$ and $\mathcal{M}_n, v \models_{IL}
(\{\psi\}, \{\chi\})$. But then, by induction hypothesis we must
have $\bigcup_{n \in \omega}\mathcal{M}_n, v \models_{IL}
(\{\psi\}, \{\chi\})$ and we also have $w\mathrel{(\bigcup_{n \in
\omega}R_n)}v$ so that $\bigcup_{n \in \omega}\mathcal{M}_n, w
\not\models_{IL} \psi \to \chi$. In the other direction, assume
that $\bigcup_{n \in \omega}\mathcal{M}_n, w \not\models_{IL} \psi
\to \chi$. Then, for some $v \in \bigcup_{n \in \omega}W_n$ such
that $w\mathrel{(\bigcup_{n \in \omega}R_n})v$ it is true that
$\bigcup_{n \in \omega}\mathcal{M}_n, v \models_{IL} (\{\psi\},
\{\chi\})$. But then, for some $k \geq n$, we must have both $w,v
\in W_k$ and $w\mathrel{R_k}v$, so that we get $\mathcal{M}_k, w
\not\models_{IL} \psi \to \chi$. By obvious transitivity of
$\preccurlyeq_{IL}$ we get then that $\mathcal{M}_n
\preccurlyeq_{IL} \mathcal{M}_k$, whence $\mathcal{M}_n, w
\not\models_{IL} \psi \to \chi$.
\end{proof}

Our main theorem is then that no proper extension of $\mathsf{IL}$
displays the combination of useful properties established in Lemma
\ref{L:properties}. In other words, we are going to establish the
following:
\begin{theorem}\label{L:main}
Let $\mathcal{L}$ be an abstract intuitionistic logic. If
$\mathsf{IL} \trianglelefteq \mathcal{L}$ and $\mathcal{L}$ is
invariant under asimulations, intuitionistically compact, and has the TUP,
then $\mathsf{IL} \equiv \mathcal{L}$.
\end{theorem}

\section{The proof of Theorem \ref{L:main}}
Before we start with the proof, we need one more piece of
notation. If $\mathcal{L}$ is an abstract intuitionistic logic,
$\Theta$ a vocabulary, and $\Gamma \subseteq L(\Theta)$, then we
let $Mod_\mathcal{L}(\Theta,\Gamma)$ denote the class of pointed
$\Theta$-models $(\mathcal{N}, u)$ such that for every $\phi \in
\Gamma$ it is true that:
$$
\mathcal{N}, u \models_\mathcal{L} \phi.
$$
If $\Gamma = \{ \phi \}$ for some $\phi \in L(\Theta)$ then
instead of $Mod_\mathcal{L}(\Theta,\Gamma)$ we simply write
$Mod_\mathcal{L}(\Theta,\phi)$.

We now start by establishing a couple of technical facts first:

\begin{proposition}\label{L:proposition1}
Let $\mathcal{L}$ be a almost strong S-closed  abstract intuitionistic
logic extending $\mathsf{IL}$. Suppose that $\mathsf{IL}
\not\equiv\mathcal{L}$. Then, there are $\phi \in L(\Theta_\phi)$
and pointed $\Theta_\phi$-models $(\mathcal{M}_1, w_1)$,
$(\mathcal{M}_2, w_2)$ such that $Th^+_{IL}(\mathcal{M}_1, w_1)
\subseteq Th^+_{IL}(\mathcal{M}_2, w_2)$ while $\mathcal{M}_1, w_1
\models_\mathcal{L} \phi$ and $\mathcal{M}_2, w_2
\not\models_\mathcal{L} \phi$.
\end{proposition}

\begin{proof}
Suppose that for an arbitrary $\phi \in L(\Theta_\phi)$ we have
shown that:

\begin{center}
\begin{itemize}
\item[(i)]  $Mod_\mathcal{L}(\Theta_\phi,\phi) =
\bigcup_{(\mathcal{N}, u) \in Mod_\mathcal{L}(\Theta_\phi,\phi)}
Mod_\mathcal{L}(\Theta_\phi,Th^+_{IL}(\mathcal{N}, u))$,

\end{itemize}
\end{center}
Let  $(\mathcal{N}, u) \in Mod_\mathcal{L}(\Theta_\phi,\phi)$ be
arbitrary. The above implies that every $\Theta_\phi$-model of
$Th^+_{IL}(\mathcal{N}, u)$ must be a model of $\phi$. But then
the theory $ (Th^+_{IL}(\mathcal{N}, u), \{ \phi \})$ is
$\mathcal{L}$-unsatisfiable. By the intuitionistic compactness of
$\mathcal{L}$, for some finite $\Psi_{(\mathcal{N}, u)} \subseteq
Th^+_{IL}(\mathcal{N}, u)$ (and we can pick a unique one using the
Axiom of Choice), the theory $(\Psi_{(\mathcal{N}, u)} ,\{ \phi\})$ is
$\mathcal{L}$-unsatisfiable. Hence, $\bigwedge \Psi_{(\mathcal{N},
u)} $ logically implies $\phi$ in $\mathcal{L}$. Then, given (i),
we get that
\begin{center}
\begin{itemize}
 \item [(ii)] $Mod_\mathcal{L}(\Theta_\phi,\phi) = \bigcup_{(\mathcal{N}, u) \in Mod_\mathcal{L}(\Theta_\phi,\phi)}Mod_\mathcal{L}(\Theta_\phi,\Psi_{(\mathcal{N}, u)})$.
\end{itemize}
\end{center}
However, this means that the theory $(\{ \phi \}, \{ \bigwedge
\Psi_{(\mathcal{N}, u)} \mid (\mathcal{N}, u) \in
Mod_\mathcal{L}(\Theta_\phi,\phi) \})$ is
$\mathcal{L}$-unsatisfiable and by (i), for some finite $\Gamma
\subseteq  \{ \bigwedge \Psi_{(\mathcal{N}, u)} \mid (\mathcal{N},
u) \in Mod_\mathcal{L}(\Theta_\phi,\phi) \}$, the theory $(\{ \phi
\}, \Gamma)$ is $\mathcal{L}$-unsatisfiable. This means that  $
Mod_\mathcal{L}(\Theta_\phi,\phi) \subseteq
Mod_\mathcal{L}(\Theta_\phi,\bigvee\Gamma)$. So, using (ii), since
clearly $Mod_\mathcal{L}(\Theta_\phi,\bigvee\Gamma) \subseteq
\bigcup_{(\mathcal{N}, u) \in
Mod_\mathcal{L}(\Theta_\phi,\phi)}Mod_\mathcal{L}(\Theta_\phi,\Psi_{(\mathcal{N},
u)})$, we get that:
\begin{center}
\begin{itemize}
 \item [(iii)] $Mod_\mathcal{L}(\Theta_\phi,\phi) = Mod_\mathcal{L}(\Theta_\phi,\bigvee\Gamma)$.
\end{itemize}
\end{center}
Now, $ \bigvee \Gamma$ is a perfectly good formula of
$IL(\Theta_\phi)$ involving only finitary conjunctions and
disjunctions. So we have shown that every $\phi \in
L(\Theta_\phi)$ is just an intuitionistic $\Theta_\phi$-formula
and hence that $\mathcal{L} \equiv \mathsf{IL}$ which is in
contradiction with the hypothesis of the proposition.

Therefore, (i) must fail for at least one $\phi \in
L(\Theta_\phi)$, and clearly, for this $\phi$ it can only fail if
$$
\bigcup_{(\mathcal{N}, u) \in Mod_\mathcal{L}(\Theta_\phi,\phi)}
Mod_\mathcal{L}(\Theta_\phi,Th^+_{IL}(\mathcal{N}, u))
\not\subseteq Mod_\mathcal{L}(\Theta_\phi,\phi).
$$
 But the latter
means that for some pointed intuitionistic $\Theta_\phi$-model
$(\mathcal{M}_1, w_1)$ such that $\mathcal{M}_1, w_1
\models_\mathcal{L}\phi$ there is another $\Theta_\phi$-model
$(\mathcal{M}_2, w_2)$ such that both $\mathcal{M}_2, w_2
\not\models_\mathcal{L}\phi$ and $Th^+_{IL}(\mathcal{M}_1, w_1)$
is satisfied at $(\mathcal{M}_2, w_2)$. The latter means, in turn,
that we have $Th^+_{IL}(\mathcal{M}_1, w_1) \subseteq
Th^+_{IL}(\mathcal{M}_2, w_2)$ as desired.
\end{proof}

Assume $\mathcal{M}$ is a $\Theta$-model. We define that
$\Theta_\mathcal{M}$ is $\Theta \cup \{ q^+_w, q^-_w \mid w \in W
\}$ such that $\Theta \cap \{ q^+_w, q^-_w \mid w \in W \} =
\varnothing$, and we define that $[\mathcal{M}] = (W, R, [V])$ is
the $\Theta_\mathcal{M}$-model, such that $W$ and $R$ are just
borrowed from $\mathcal{M}$ and $[V]$ coincides with $V$ on
elements of $\Theta$, whereas for arbitrary $v, w \in W$ we set
that  $v \in [V](q^+_w)$ iff $w\mathrel{R}v$ and $v \notin
[V](q^-_w)$ iff $v\mathrel{R}w$.

\begin{lemma}\label{L:lemma1}
Let $\mathcal{L}$ be a bisimulation invariant abstract
intuitionistic logic extending $\mathsf{IL}$, let $(\mathcal{M},
w)$ be a pointed $\Theta$-model, and let $(\mathcal{N}, v)$
another pointed $\Theta_{\mathcal{M}^{un}_w}$-model. Assume that
$Th_\mathcal{L}(\mathcal{N}, v) =
Th_\mathcal{L}([\mathcal{M}^{un}_w], w)$. Then there exists an
$\mathcal{L}$-embedding $f$ of $[\mathcal{M}^{un}_w]$ into
$\mathcal{N}$ such that $f(w) = v$.
\end{lemma}

\begin{proof} Consider the relation $Z$ relating $w$ only to $v$,
and every other element $\bar{u}_k \in W^{un}_w$ to the worlds in
$\mathcal{N}$ that $\mathcal{L}$-satisfy
$Th_\mathcal{L}([\mathcal{M}^{un}_w], \bar{u}_k)$. We prove that
such $Z$ is a total simulation from $([\mathcal{M}^{un}_w], w)$ to
$(\mathcal{N}, v)$. Indeed, conditions \eqref{E:c2} and
\eqref{E:c3} are obviously satisfied. We treat condition
\eqref{E:c4}.

Assume that $\bar{v}_k\mathrel{Z}v'$ and that
$\bar{v}_k\mathrel{R^{un}_w}\bar{v}_n$ (then $n \geq k$). By
definition of $Z$ we have then $Th_\mathcal{L}(\mathcal{N}, v') =
Th_\mathcal{L}([\mathcal{M}^{un}_w], \bar{v}_k)$, and also, since
$w$ is the root of $\mathcal{M}^{un}_w$, we have, by Lemma
\ref{L:monotonicity} and invariance of $\mathcal{L}$ under
asimulations that
$$
Th^+_\mathcal{L}([\mathcal{M}^{un}_w], w) \subseteq
Th^+_\mathcal{L}([\mathcal{M}^{un}_w], \bar{v}_k) =
Th^+_\mathcal{L}(\mathcal{N}, v').
$$
By definition of $[\mathcal{M}^{un}_w]$ and closure of
$\mathcal{L}$ w.r.t. intuitionistic implication, we have that the
set:
$$
\Gamma_{\bar{v}_n} = \{ q^+_{\bar{v}_n} \to \phi \mid \phi \in
Th^+_\mathcal{L}([\mathcal{M}^{un}_w], \bar{v}_n) \} \cup \{ \psi
\to q^-_{\bar{v}_n} \mid \psi \in
Th^-_\mathcal{L}([\mathcal{M}^{un}_w], \bar{v}_n) \}
$$
is a subset of $Th^+_\mathcal{L}([\mathcal{M}^{un}_w], w)$, thus
also of $Th^+_\mathcal{L}(\mathcal{N}, v')$. We also know that, by
$\bar{v}_k\mathrel{R^{un}_w}\bar{v}_n$, and by the fact that the
theory $(\{ q^+_{\bar{v}_n} \}\{ q^-_{\bar{v}_n}\})$ is
$\mathcal{L}$-satisfied at $([\mathcal{M}^{un}_w], \bar{v}_n)$, we
have:
 $$
 [\mathcal{M}^{un}_w], \bar{v}_k \not\models_\mathcal{L}
q^+_{\bar{v}_n}\to q^-_{\bar{v}_n}.
$$

 Therefore, $\mathcal{N}, v'
\not\models_\mathcal{L} q^+_{\bar{v}_n}\to q^-_{\bar{v}_n}$. So
take some $v'' \in U$ such that $v'\mathrel{R^\mathcal{N}}v''$ and $(\{
q^+_{\bar{v}_n} \}\{ q^-_{\bar{v}_n}\})$ is
$\mathcal{L}$-satisfied at $(\mathcal{N}, v'')$. By
$\Gamma_{\bar{v}_n}  \subseteq Th^+_\mathcal{L}(\mathcal{N}, v')$,
we also have that $\Gamma_{\bar{v}_n}  \subseteq
Th^+_\mathcal{L}(\mathcal{N}, v'')$, so that $
Th_\mathcal{L}(\mathcal{N}, v'') =
Th_\mathcal{L}([\mathcal{M}^{un}_w], \bar{v}_n)$, whence we get
that $\bar{v}_n\mathrel{Z}v''$ and condition \eqref{E:c4} is
verified.

By Lemma \ref{L:homomorphism}, there must be then a homomorphism
$f$ from $[\mathcal{M}^{un}_w]$ to $\mathcal{N}$ such that
$\bar{v}_n\mathrel{Z}f(\bar{v}_n)$ for all $\bar{v}_n \in
W^{un}_w$, so, in particular, $w\mathrel{Z}f(w)$  but by
definition $Z$ only relates $w$ to $v$, hence $f(w)= v$.

For this $f$, we get condition \eqref{E:c3e} immediately by
definition of $Z$. To establish the injectivity of $f$ assume that
$\bar{v}_n,\bar{u}_k \in W^{un}_w$ are such that $\bar{v}_n \neq
\bar{u}_k$. Then, by Lemma \ref{unravellinglemma}, either
$\bar{v}_n$ is not an $R^{un}_w$-successor of $\bar{u}_k$ or
$\bar{u}_k$ is not an $R^{un}_w$-successor of $\bar{v}_n$. If the
first, $f(\bar{u}_k)$ $\mathcal{L}$-satisfies $q^+_{\bar{u}_k}$ in
$\mathcal{N}$, while $f(\bar{v}_n)$ does not, so $f(\bar{u}_k)
\neq f(\bar{v}_n)$. If the second, then a symmetric thing happens
with $q^+_{\bar{v}_n}$.

Finally, if, for some $\bar{v}_n, \bar{u}_k \in W^{un}_w$,
$f(\bar{v}_n)$ is a successor of $f(\bar{u}_k)$ in $\mathcal{N}$,
we must have that $\mathcal{N}, f(\bar{v}_n) \vDash_\mathcal{L}
q^+_{\bar{u}_k}$, so that $[\mathcal{M}^{un}_w],
\bar{v}_n\vDash_\mathcal{L} q^+_{\bar{u}_k}$, which means by
definition of the valuation of $q^+_{\bar{u}_k}$ that $\bar{v}_n$
is an $R^{un}_w$-successor of $\bar{u}_k$ in
$[\mathcal{M}^{un}_w]$ so that condition \eqref{E:c4e} is verified
as well.
 \end{proof}

\begin{proposition}\label{L:saturation}
Let $(\mathcal{M}, w)$ be a pointed $\Theta$-model and let
$\mathcal{L}$ be an abstract intuitionistic logic which is
invariant under asimulations, intuitionistically compact, has TUP and extends
$\mathsf{IL}$. Then there is an $\mathcal{L}$-saturated pointed
$\Theta$-model $(\mathcal{N}, s)$ such that $w \in U$ and
$Th_\mathcal{L}(\mathcal{M}, w) = Th_\mathcal{L}(\mathcal{N}, w)$.
\end{proposition}
\begin{proof}
Consider $(\mathcal{M}^{un}_w, w)$. By Lemma
\ref{unravellinglemma}.3, $(\mathcal{M}^{un}_w, w)$ is bisimilar
to $(\mathcal{M}, w)$, therefore, by Lemma
\ref{L:bisimulationinvariance} and invariance of $\mathcal{L}$
under asimulations, we get that $Th_\mathcal{L}(\mathcal{M}, w) =
Th_\mathcal{L}(\mathcal{M}^{un}_w, w)$. Consider then
$[\mathcal{M}^{un}_w]$ and assume $\bar{v}_k \in W^{un}_w$. For a
given $\bar{v}_k \in W^{un}_w$, we will denote the family of all
$L(\Theta_{[\mathcal{M}^{un}_w]})$-theories which are finitely
$\mathcal{L}$-satisfiable in $[\mathcal{M}^{un}_w]$ by successors
of $\bar{v}_k$ by $FinSat_{\bar{v}_k}$. For every $\bar{v}_k \in
W^{un}_w$ and every $(\Gamma, \Delta) \in FinSat_{\bar{v}_k}$ we
add to $\Theta_{[\mathcal{M}^{un}_w]}$ a fresh pair of
propositional letters $r^+_\Gamma$ and $r^-_\Delta$. Call the
resulting vocabulary $\Theta'$. Consider next the following
$L(\Theta')$-theory $(\Xi,Th^-_\mathcal{L}([\mathcal{M}^{un}_w],
w))$, where:
\begin{align*}
\Xi = &Th^+_\mathcal{L}([\mathcal{M}^{un}_w], w) \cup\\
&\cup  \{r^+_\Gamma \to \phi, \psi \to r^-_\Delta, (r^+_\Gamma \to
r^-_\Delta) \to q^-_{\bar{v}_k}\mid (\Gamma, \Delta) \in
FinSat_{\bar{v}_k}, \phi \in \Gamma, \psi \in \Delta, \bar{v}_k
\in W^{un}_w \},
\end{align*}
$(\Xi,Th^-_\mathcal{L}([\mathcal{M}^{un}_w], w))$ is itself
finitely $\mathcal{L}$-satisfiable, since given finite $\Xi_0
\subseteq \Xi$ and $\Omega_0 \subseteq
Th^-_\mathcal{L}([\mathcal{M}^{un}_w], w)$, we know wlog that:
$$
\Xi_0 = \Gamma \cup T(S_{(\Gamma_1, \Delta_1)}) \cup, \ldots, \cup
T(S_{(\Gamma_n, \Delta_n)}),
$$
where $\Gamma$ is a finite subtheory of
$Th^+_\mathcal{L}([\mathcal{M}^{un}_w], w)$, $(\Gamma_1, \Delta_1)
\in FinSat_{\bar{v}_{k_1}},\ldots, (\Gamma_n, \Delta_n)\in
FinSat_{\bar{v}_{k_n}}$, with $\bar{v}_{k_1},\ldots,
\bar{v}_{k_n}\in W^{un}_w$, for every $1 \leq i \leq n$
$S_{(\Gamma_i, \Delta_i)}$ is a finite subset of $(\Gamma_i,
\Delta_i)$ and, again for every $1 \leq i \leq n$:
$$
T(S_{(\Gamma_i, \Delta_i)}) = \{r^+_{\Gamma_i} \to \phi, \psi \to
r^-_{\Delta_i}, (r^+_{\Gamma_i} \to r^-_{\Delta_i}) \to
q^-_{\bar{v}_{k_i}}\mid \phi, \psi \in S_{(\Gamma_i, \Delta_i)}\}.
$$

By definition of $FinSat$, every $S_{(\Gamma_i, \Delta_i)}$ for $1
\leq i \leq n$ will be $\mathcal{L}$-satisfied in some
$v(S_{(\Gamma_i, \Delta_i)}) \in W^{un}_w$ such that
$\bar{v}_{k_i}\mathrel{R^{un}_w}v(S_{(\Gamma_i, \Delta_i)})$. It
is easy to see then, that every such $(\Xi_0,\Omega_0)$ will be
$\mathcal{L}$-satisfied in the extension of
$([\mathcal{M}^{un}_w], w)$, where for every $1 \leq i \leq n$ the
propositional letters $r^+_{\Gamma_i}, r^-_{\Delta_i}$ are
identified with $q^+_{v(S_{(\Gamma_i, \Delta_i)})},
q^-_{v(S_{(\Gamma_i, \Delta_i)})}$, respectively.

Therefore, by the intuitionistic compactness of $\mathcal{L}$,
$(\Xi,Th^-_\mathcal{L}([\mathcal{M}^{un}_w], w))$ itself is
$\mathcal{L}$-satisfiable. Let $(\mathcal{M}'_1, w_1)$ be a
pointed $\Theta'$-model $\mathcal{L}$-satisfying
$(\Xi,Th^-_\mathcal{L}([\mathcal{M}^{un}_w], w))$, and we may
assume that $w_1$ is the root of $\mathcal{M}'_1$ (if not, throw
away every world which is not accessible from $w_1$ in
$\mathcal{M}'_1$). Now, let $\mathcal{M}_1$ be the reduct of
$\mathcal{M}'_1$ to $\Theta_{[\mathcal{M}^{un}_w]}$. We know that
$(\mathcal{M}_1, w_1)$ $\mathcal{L}$-satisfies
$Th_\mathcal{L}([\mathcal{M}^{un}_w], w)$, therefore, by Lemma
\ref{unravellinglemma}.3, $((\mathcal{M}_1)^{un}_{w_1}, w_1)$ also
$\mathcal{L}$-satisfies $Th_\mathcal{L}([\mathcal{M}^{un}_w], w)$
and by Lemma \ref{L:lemma1}, there must be an embedding $f$ of
$[\mathcal{M}^{un}_w]$ into $(\mathcal{M}_1)^{un}_{w_1}$ with
$f(w) = w_1$. We now prove the following:

\emph{Claim}. If $\bar{v}_k \in W^{un}_w$ and a
$L(\Theta_{[\mathcal{M}^{un}_w]})$-theory $(\Gamma, \Delta)$ is
finitely $\mathcal{L}$-satisfiable in $[\mathcal{M}^{un}_w]$ by
successors of $\bar{v}_k$, then there is an $u \in
(W_1)^{un}_{w_1}$ such that
$f(\bar{v}_k)\mathrel{(R_1)^{un}_{w_1}}u$ and $(\Gamma, \Delta)$
is $\mathcal{L}$-satisfied at $((\mathcal{M}_1)^{un}_{w_1}, u)$.

Indeed, we have $(\Gamma, \Delta) \in FinSat_{\bar{v}_k}$, and we
also have $[\mathcal{M}^{un}_w], \bar{v}_k \not\models_\mathcal{L}
q^-_{\bar{v}_k}$, hence $(\mathcal{M}_1)^{un}_{w_1}, f(\bar{v}_k)
\not\models_\mathcal{L} q^-_{\bar{v}_k}$. Since
$(\mathcal{M}_1)^{un}_{w_1}$ is an intuitionistic unravelling,
$f(\bar{v}_k)$ is in fact some $\bar{u}_r \in (W_1)^{un}_{w_1}$.
By Lemma \ref{unravellinglemma}.4 we must have then
$\mathcal{M}_1, u_r \not\models_\mathcal{L} q^-_{\bar{v}_k}$,
therefore, by Expansion, $\mathcal{M}'_1, u_r
\not\models_\mathcal{L} q^-_{\bar{v}_k}$. Since $w_1$ is the root
of $\mathcal{M}'_1$ and
$(\Xi,Th^-_\mathcal{L}([\mathcal{M}^{un}_w], w))$ is
$\mathcal{L}$-satisfied at $(\mathcal{M}'_1, w_1)$, we must also
have $\mathcal{M}'_1, u_r \models_\mathcal{L} \Xi$, and, in
particular, $\mathcal{M}'_1, u_r \models_\mathcal{L} (r^+_\Gamma
\to r^-_\Delta) \to q^-_{\bar{v}_k}$. Therefore, there must be a
successor $u'$ to $u_r$ in $\mathcal{M}'_1$ such that $(\{
r^+_\Gamma \},\{ r^-_\Delta\})$ is $\mathcal{L}$-satisfied at
$(\mathcal{M}'_1, u')$. Since we will also have $\mathcal{M}'_1,
u' \models_\mathcal{L} \Xi$, this means that $(\Gamma, \Delta)$
will be $\mathcal{L}$-satisfied at $(\mathcal{M}'_1, u')$ as well.
Again by Expansion, we get that $(\Gamma, \Delta)$ will be
$\mathcal{L}$-satisfied at $(\mathcal{M}_1, u')$ and further, by
Lemma \ref{unravellinglemma}.4 , that $(\Gamma, \Delta)$ will be
$\mathcal{L}$-satisfied at $((\mathcal{M}_1)^{un}_{w_1},
(\bar{u}_r, u'))$. By $u_r\mathrel{R_1}u'$, we have that
$(\bar{u}_r, u') \in (W_1)^{un}_{w_1}$ and
$\bar{u}_r\mathrel{(R_1)^{un}_{w_1}}(\bar{u}_r, u')$. Therefore,
setting $u := (\bar{u}_r, u')$, we get our Claim verified.

Observe further, that by the remark after Definition
\ref{D:embedding}, the existence of the above-defined embedding $f$
means that we may assume that $[\mathcal{M}^{un}_w]
\preccurlyeq_\mathcal{L}(\mathcal{M}_1)^{un}_{w_1}$ and identify
$w_1$ with $w$. Repeating this construction $\omega$ times and
taking $\Theta_{[\mathcal{M}^{un}_w]}$-reducts of the resulting
models so that all of them are in the same vocabulary, we obtain
an infinite chain of $\mathcal{L}$-elementary submodels of the
following form:
$$
[\mathcal{M}^{un}_w] \preccurlyeq_\mathcal{L}
(\mathcal{M}_1)^{un}_{w}\preccurlyeq_\mathcal{L},\ldots,
\preccurlyeq_\mathcal{L}
(\mathcal{M}_n)^{un}_{w}\preccurlyeq_\mathcal{L},\ldots,
$$
with $w$ being the root of every model in the chain. To simplify
notation and achieve the uniformity of subscripts, we rename
$[\mathcal{M}^{un}_w]$ as $\mathcal{N}_0$, and for $i \geq 1$,
rename $(\mathcal{M}_i)^{un}_{w}$ as  $\mathcal{N}_i$. We consider
then $\bigcup_{i \in \omega}\mathcal{N}_i$. Since $\mathcal{L}$
has TUP, we know that for every $j \in \omega$ we have
\begin{equation}\label{E:equ1}
\mathcal{N}_j \preccurlyeq_\mathcal{L} \bigcup_{i \in
\omega}\mathcal{N}_i.
\end{equation}
We claim that  $\bigcup_{i \in \omega}\mathcal{N}_i$ is an
$\mathcal{L}$-saturated $\Theta_{[\mathcal{M}^{un}_w]}$-model.
Indeed, assume that $(\Gamma, \Delta)$ is an
$L(\Theta_{[\mathcal{M}^{un}_w]})$-theory which is finitely
$\mathcal{L}$-satisfiable among successors of some $v$ in
$\bigcup_{i \in \omega}\mathcal{N}_i$. Then for any finite
$(\Gamma_0, \Delta_0) \subseteq (\Gamma, \Delta)$ we will have
$$
\bigcup_{i \in \omega}\mathcal{N}_i, v \not\models_\mathcal{L}
\bigwedge\Gamma_0 \to \bigvee \Delta_0.
$$
Then choose $j \in \omega$ such that $\mathcal{N}_j$ is the first
model in the chain, where $v$ occurs.  By \eqref{E:equ1}, we will
have:
$$
\mathcal{N}_j, v \not\models_\mathcal{L} \bigwedge\Gamma_0 \to
\bigvee \Delta_0
$$
for any finite $(\Gamma_0, \Delta_0) \subseteq (\Gamma, \Delta)$.
Therefore, $(\Gamma, \Delta)$ will be finitely
$\mathcal{L}$-satisfiable in $\mathcal{N}_j$ by successors of $v$
and by the respective version of our Claim above, this means that
there is a successor $v'$ to $v$ in $\mathcal{N}_{j + 1}$ such
that $(\Gamma, \Delta)$ is $\mathcal{L}$-satisfied at
$(\mathcal{N}_{j + 1}, v')$. But then, again by \eqref{E:equ1}, we
get that $(\Gamma, \Delta)$ is also $\mathcal{L}$-satisfied at
$(\bigcup_{i \in \omega}\mathcal{N}_i, v')$.

Since $\bigcup_{i \in \omega}\mathcal{N}_i$ is therefore shown to
be an $\mathcal{L}$-saturated
$\Theta_{[\mathcal{M}^{un}_w]}$-model, then setting $\mathcal{N}$
to be $\Theta$-reduct of $\bigcup_{i \in \omega}\mathcal{N}_i$, we
immediately get that $\mathcal{N}$ is $\mathcal{L}$-saturated.
Also, by \eqref{E:equ1} and the fact that $\mathcal{M}^{un}_w$ is
the $\Theta$-reduct of $[\mathcal{M}^{un}_w]$ we get that:
$$
Th_\mathcal{L}(\mathcal{M}^{un}_w, w) =
Th_\mathcal{L}(\mathcal{N}, w).
$$
Therefore, by Lemma \ref{unravellinglemma}.4, we get that
$Th_\mathcal{L}(\mathcal{M}, w) = Th_\mathcal{L}(\mathcal{N}, w)$
and thus we are done.
\end{proof}

We are now in a position to prove Theorem \ref{L:main}. Indeed,
assume the hypothesis of the theorem, and assume, for
contradiction, that $\mathcal{L}\not\equiv\mathsf{IL}$. By
Proposition \ref{L:proposition1}, there must be $\phi \in
L(\Theta_\phi)$ and pointed intuitionistic $\Theta_\phi$-models
$(\mathcal{M}_1, w_1)$, $(\mathcal{M}_2, w_2)$ such that
$Th^+_{IL}(\mathcal{M}_1, w_1) \subseteq Th^+_{IL}(\mathcal{M}_2,
w_2)$ while $\mathcal{M}_1, w_1 \models_\mathcal{L} \phi$ and
$\mathcal{M}_2, w_2 \not\models_\mathcal{L} \phi$. By Proposition
\ref{L:saturation}, take $\mathcal{L}$-saturated
$\Theta_\phi$-models $\mathcal{N}_1$ and $\mathcal{N}_2$ such that
$w_i \in U_i$ and $Th_\mathcal{L}(\mathcal{M}_i, w_i) =
Th_\mathcal{L}(\mathcal{N}_i, w_i)$ for $i \in \{ 0, 1 \}$. We
will have then, of course, that $\mathcal{N}_1, w_1
\models_\mathcal{L} \phi$, but $\mathcal{N}_2, w_2
\not\models_\mathcal{L} \phi$. On the other hand, we will still
have $Th^+_{IL}(\mathcal{N}_1, w_1) \subseteq
Th^+_{IL}(\mathcal{N}_2, w_2)$, whence by Corollary
\ref{L:asimulationscorollary} there must be an asimulation $A$
from $(\mathcal{N}_1, w_1)$ to $(\mathcal{N}_2, w_2)$, but then,
since $\mathcal{L}$ is invariant under asimulations, we must also
have $Th^+_\mathcal{L}(\mathcal{N}_1, w_1) \subseteq
Th^+_\mathcal{L}(\mathcal{N}_2, w_2)$. Now, since $\phi \in
Th^+_\mathcal{L}(\mathcal{N}_1, w_1)$, we can see that $\phi \in
Th^+_\mathcal{L}(\mathcal{N}_2, w_2)$, so that $\mathcal{N}_2, w_2
\models_\mathcal{L} \phi$, which is a contradiction.

\section{A topological approach to the main theorem}\label{S:Top}

There is a topological reading of our result. This sort of take on Lindstr\"om theorems has been explored by some authors, in particular, Caicedo (see \cite{caicedo}). One of the fundamental interests of such an approach is that logical compactness  properties become topological compactness properties. To see this, for each vocabulary $\Theta$, we must first define a topology (to be called the \emph{intuitionistic topology for $\Theta$} $-$indeed, the reference to $\Theta$ will be dropped when the context makes it clear that some   $\Theta$ has been fixed) on the the space of all intuitionistic pointed $\Theta$-models (call it $S$ for the purposes of this section). Abstract logics will then roughly correspond to topologies on such space. We will assume that the reader is familiar with the basic notions of general topology that can be acquired in a place like \cite{kelley}.

 We start by introducing a closed base $\mathsf{B}$ for our topology. The elements of $\mathsf{B}$ have the form $Mod(\phi)$
 for some formula $\phi$ in the usual language $\{\rightarrow, \wedge, \vee, \bot, \top\}$ of intuitionistic propositional logic. $\mathsf{B}$
  is a closed base for a topology on $S$  because for any two $A, B \in \mathsf{B}$ there is $W \in \mathsf{B}$ such that $A \cup B = W$:
   given $Mod(\phi)$ and $Mod(\psi)$, consider $Mod(\phi \vee \psi)$. The intuitionistic topology is just the topology induced by this base,
    that is, take as closed collections arbitrary intersections of elements of $\mathsf{B}$.

Now we may observe that we ended up with a topology where the
closed collections have the form $Mod(T)$ for some intuitionistic
set of formulas $T$. Denote $2^{IL(\Theta)}$ with $\wp(Fmla)$. In
particular, given a subclass $A$ of the space, if we denote by
$Th^+_{IL}(A)$ the collection of all intuitionistic formulas
holding in every member of $A$, then  $A^- $, the closure of $A$
in our topology, is just $\bigcap_{T \in \wp(Fmla) \, \atop{A
\subseteq Mod(T)}} Mod(T) = Mod(Th^+_{IL}(A)) $.

The space is not \emph{normal} (i.e., it is not the case that for any two disjoint closed classes $A,B$ there exist disjoint open classes $U,V$ containing $A$ and $B$ respectively), for let $Mod(T) \cap Mod(U) = \emptyset$ for two intuitionistic theories such that $Mod(T) \neq \emptyset $ and $Mod(U) \neq \emptyset $. Moreover, suppose that  $Mod(T) \subseteq S \setminus Mod(T^{\prime})$, and $Mod(U) \subseteq S \setminus Mod(U^{\prime})$ for some intuitionistic theories $T^{\prime}, U^{\prime}$.  Now take any structure $(\mathcal{M}, w) \in S \setminus Mod(T^{\prime})$ and  $(\mathcal{N}, v) \in S \setminus Mod(U^{\prime})$, but this means that for some $\phi \in \ T^{\prime}$ and $\psi \in U^{\prime}$ we have that $\phi$ and $\psi$  fail at $(\mathcal{M}, w)$ and $(\mathcal{N}, v)$ respectively, so neither $\phi$ nor $\psi$ are theorems of intuitionistic logic. But intuitionistic logic has the disjunction property ($\vdash \phi \vee \psi$ only if either $\vdash \phi$ or $\vdash \psi$), so in fact there is some $(\mathcal{M}^{\prime}, w^{\prime}) \in S$ such that $\phi \vee \psi$ fails, which implies that $(\mathcal{M}^{\prime}, w^{\prime}) \in (S \setminus Mod(T^{\prime})) \cap (S \setminus Mod(U^{\prime}))$.

The space is not  \emph{regular} (i.e., it is not the case that
closed classes and exterior points may be separated by disjoint
open classes). Take  any $Mod(T_0) \neq \emptyset$ and
$(\mathcal{M}, w) \notin Mod(T_0)$, i.e., $\mathcal{M}, w
\not\models_{IL} \phi$ for some $\phi \in T_0$. Consider further
arbitrary open classes   $S \setminus Mod(T_1)$ and $S \setminus
Mod(T_2)$ such that $Mod(T_0) \subseteq S \setminus Mod(T_1)$ and
$(\mathcal{M}, w) \in S \setminus Mod(T_2)$. Hence, taking
$(\mathcal{N}, v) \in Mod(T_0) $, it must be in $S \setminus
Mod(T_1)$. So there are $\psi \in T_1$ and $\theta \in T_2$ such
that $\mathcal{N}, v \not\models_{IL} \psi$ and $\mathcal{M}, w
\not\models_{IL} \theta$. By the disjunction property, we have
$(\mathcal{C}, u) $ falsifying both $\psi$ and $\theta$, hence the
intersection  $(S \setminus Mod(T_1))  \cap (S \setminus
Mod(T_2))$ cannot be empty.

The space is compact because each family of closed sets which has the finite intersection property has a non-empty intersection (this is just the content of the compactness theorem in intuitionistic logic $-$hence Rasiowa's name for this theorem), which is equivalent to saying that every covering of the space can be reduced to a finite subcovering.

Moreover, the space is connected. To see this we again use the disjunction property of intuitionistic logic. For suppose that $S = A \cup B$ for separated subsets $A, B \neq \emptyset$. But then $Mod(Th^+_{IL}(A)) \cap B = \emptyset$ and  $Mod(Th^+_{IL}(B)) \cap A = \emptyset$. Now take $(\mathcal{M}, w) \in A$, so it must be the case that  some $\phi \in Th^+_{IL}(A)$ fails in $(\mathcal{M}, w)$. On the other hand, taking  $(\mathcal{N}, v) \in B$, we obtain an $(\mathcal{N}, v)$ where some $\psi \in Th^+_{IL}(B)$ fails. By the disjunction property, there is a structure $(\mathcal{M}^{\prime}, w^{\prime})$ where both $\phi$ and $\psi$ fail, which means that $(\mathcal{M}^{\prime}, w^{\prime})$ is in neither the closure of $A$ nor the closure of $B$, which is impossible given that $S = A \cup B$.

The space is not Hausdorff because given two models satisfying exactly the same intuitionistic formulas (that is, they belong to exactly the same closed classes of the topology) but which are not isomorphic, so they are distinct, we have a violation of the definition of a Hausdorff  space: given two distinct points in the space we can find disjoint neighborhoods of each of them. In fact, we get with this example that the space is not $T_1$: there are distinct points $x, y$ such that for every open classes $A, B$ of the intuitionistic topology, $x \in A$ only if $y \in A$ and $y \in B$ only if $x \in B$.  Even worse, the space is not even $T_0$: by the same example we have got points $x, y$ such that there is no neighborhood of $x$ where $y$ doesn't belong. Naturally, the latter observation suffices to refute the earlier properties.

Now we look at  a related topology that plays an
important role in our main result. Consider the collection
$\mathsf{SB^{*}}$ containing all collections of the form $S
\setminus Mod(\phi, \bot)$ and $S \setminus Mod(\top,\psi)$ for
all intutionistic formulas $\phi$ and $\psi$. This forms an open
subbase for a topology on $S$.  The open base  $\mathsf{B^{*}}$
generated from $\mathsf{SB^{*}}$ is just all classes of the form
$Mod(\phi, \psi)$ for intuitionistic formulas $\phi$ and $\psi$
(this collection simply contains all finite intersections of
members of $\mathsf{SB^{*}}$). Open classes in this topology are
of the form $\bigcup_{\phi \in \Phi, \psi \in\Psi} Mod(\phi,
\psi)$ for some collections $\Phi, \Psi$ of formulas.

A topological space is said to be \emph{strongly S-closed} if every family of open sets with the finite intersection property has a non-empty intersection \cite{don}. Moreover, we will say that a space is  \emph{almost strongly S-closed} if every family of \emph{basic} open sets with the finite intersection property has a non-empty intersection. 
The topology is almost strongly S-closed this follows from the fact that every cover of $S$ by members of the subbase has a finite subcover . The latter comes from the  intuitionistic compactness
 of intuitionistic logic. Moreover, the topology is uniform, we can obtain it from what we will call (borrowing from \cite{caicedo}) the \emph{canonical uniformity}, namely the uniformity generated by the base formed by  the subclass $ U \subseteq S \times S$ containing:

$$
U_{\Phi} = \{((\mathcal{M}, w), (\mathcal{N}, v)) \mid
\mathcal{M}, w \models_{IL} \phi \ \mbox{iff} \ \mathcal{N}, v
\models_{IL} \phi, \phi \in \Phi \}
$$
for each \emph{ finite} collection $\Phi$ of formulas. Now, the
canonical uniformity generates the previously defined topology of
intuitionistic theories. To see this, first note that
$U_{\Phi}[(\mathcal{M}, w)] = \{(\mathcal{N}, v) \mid
((\mathcal{M}, w), (\mathcal{N}, v)) \in U_\Phi \}$ coincides with
$$
(S\setminus \{(\mathcal{N}, v) \mid  \mathcal{N}, v \models_{IL}
(\bigvee \Phi^-, \bot)\}) \cap (S\setminus \{(\mathcal{N}, v) \mid
\mathcal{N}, v \models_{IL} (\top, \bigwedge \Phi^+)\})
$$
 where
$\Phi^+ = \{\phi \in \Phi \mid \mathcal{M}, w \models_{IL} \phi\}$
while $\Phi^- = \{\phi \in \Phi \mid  \mathcal{M}, w
\not\models_{IL} \phi\}$. Now,  given an open class $\bigcup_{\phi
\in \Phi, \psi \in\Psi} Mod(\phi, \psi)$, if $(\mathcal{M}, w) \in
\bigcup_{\phi \in \Phi, \psi \in\Psi} Mod(\phi, \psi)$, then for
some $\phi, \psi$, we have that $(\mathcal{M}, w) \in Mod(\phi,
\psi)$. But then  $U_{\{\phi, \psi\}}[(\mathcal{M}, w)] \subseteq
Mod(\phi, \psi)$. Hence, open classes of the  previously
considered topology are open in the topology of the canonical
uniformity. Conversely, if $O$ is an open class of the topology of
the canonical uniformity, we can see that $O =
\bigcup_{(\mathcal{M}, w) \in O, \atop{U_{\Phi}[(\mathcal{M}, w)]
\subseteq O, \atop{U_{\Phi} \in U}}} U_{\Phi}[(\mathcal{M}, w)]$.
But the latter is a union of open classes in the previously
considered topology, as we have observed above.

Another space that may be interesting to explore is the quotient
space obtained from the following equivalence relation on $S$: $x
\approx y$ iff $x, y$ belong to exactly the same closed classes of
the intuitionistic topology. Now consider the quotient topology
for the natural projection $\pi: S \longrightarrow S \setminus
\approx$. This topology takes as the closed classes  all those
subclasses $U$ of $S$ such that $\pi^{-1}[U]$ is closed in the
intuitionistic topology.  So the closed classes in this topology
have the form $\{[(\mathcal{M}, w)] \mid (\mathcal{M}, w) \in
Mod(T) \}$ for some intuitionistic set of formulas $T$.

The  quotient space is $T_0$:  for if $[(\mathcal{M}, w)] \neq
[(\mathcal{N}, v)]$ then we have some $\phi$ such that
$(\mathcal{M}, w) \in Mod( \phi)$ while $(\mathcal{N}, v) \notin
Mod( \phi)$, so $[(\mathcal{M}, w)] \in \{[(\mathcal{M}, w)] \mid
(\mathcal{M}, w) \in Mod(\phi) \}$ whereas $[(\mathcal{N}, v)]
\notin \{[(\mathcal{N}, v)] \mid (\mathcal{ N}, v) \in Mod(\phi)
\}$, so there is a neighborhood of one of the points not
containing the other.

On the other hand, the space is not Hausdorff: take
$[(\mathcal{M}, w)] \neq [(\mathcal{N}, v)]$, and notice that for
every open classes $U, V$ such that $[(\mathcal{M}, w)] \in U$ and
$[(\mathcal{N}, v)] \in V$ we have that, for some sets of intuitionistic
formulas $T, O$, we have that $U =  (S \setminus \approx)
\setminus \{[(\mathcal{C}, u)] \mid (\mathcal{C}, u) \in Mod(T)
\}$ and $V = (S \setminus \approx) \setminus \{[(\mathcal{C}, u)]
\mid (\mathcal{C}, u) \in Mod(O) \}$. So $(\mathcal{M}, w) \notin
Mod (\phi)$ and $(\mathcal{N}, v) \notin Mod(\psi)$ for some $\phi
\in T$ and $\psi \in U$, so using the disjunction property, we get
some $(\mathcal{M}^{\prime}, w^{\prime})$ such that
$(\mathcal{M}^{\prime}, w^{\prime}) \notin Mod(\phi \vee \psi)$,
which means that  $[(\mathcal{M}^{\prime}, w^{\prime})] \in U \cap
V$.

Finally, the main theorem of this paper may be read as: \emph{for any vocabulary, the intuitionistic topology is the finest topology such that (i) its closed sets are closed under asimulations, (ii) it has the Tarski-Union property and (iii) the subbase of the topology of the canonical uniformity is such that every open cover of the space by members of the subbase has a finite subcover.}

\section{Conclusion}

We have succeeded in establishing a Lindstr\"om characterization of intuitionistic propositional logic over intuitionistic models involving the properties of  intuitionistic compactness, the TUP and preservation under asimulations. In future work we will show how this characterization can be extended to the case of predicate intuitionistic logic. Our work, of course, opens a world of new questions. A particularly immediate one is to find some other combination of interesting model theoretic properties that implies the three mentioned properties and would therefore give us a new intuitionistic Lindstr\"om theorem.

On the other hand, it would be nice to work out the details of what is the algebraic content of our result via duality theory. Pointed intuitionistic models are dual to Heyting matrices (that is,  pairs formed by a Heyting algebra and a filter on such algebra), so the work reduces to investigate the algebraic counterparts of the TUP and the preservation under asimulations.

\section*{Acknowledgments}
We are grateful to the anonymous referee and the editor of this journal for their time and effort on this paper. This article is a testimony to the eternal indestructible friendship between the Cuban and Soviet peoples. 

Guillermo Badia is supported by the project I 1923-N25 of the Austrian Science Fund (FWF). 
Grigory Olkhovikov is supported by Deutsche
Forschungsgemeinschaft (DFG), project WA 936/11-1.

}
\end{document}